\documentclass{amsart}
\usepackage{amsmath, amsthm, amssymb}
\usepackage{epsfig}
\usepackage{rawfonts}
\usepackage{enumerate}
\usepackage{graphics}

\usepackage{xspace}
\usepackage{graphicx}

\usepackage{pgf,tikz,pgfplots}
\usepackage{mathrsfs}
\usetikzlibrary{arrows}
\usepackage{amsmath}
\usepackage{amsfonts}
\usepackage{amssymb}
\usepackage{amsthm}
\usepackage{graphicx}
\usepackage{booktabs}
\usepackage{caption}
\usepackage{listings}
\usepackage{setspace}
\usepackage[mathscr]{eucal}
\usepackage{pgfplots}
\usepackage{hyperref}
\usepackage{wrapfig}
\usepackage{floatflt,epsfig}
\usepackage{ dsfont }
\usepackage{amscd}
\usepackage{tikz-cd}
\usepackage{fancyhdr}
\usepackage[all]{xy}
\usepackage{latexsym}
\usepackage{amscd}
\usepackage{pifont}
\usepackage{eufrak}
\usepackage{subfig}

\newcommand{\numberset}{\mathbb}
\newcommand{\N}{\numberset{N}}
\newcommand{\Z}{\numberset{Z}}

\def\NN{{\mathbb N}}
\def\ZZ{{\mathbb Z}}

\newcommand{\cC}{\mathcal{C}}

\newcommand{\cP}{\mathcal{P}}
\newcommand{\cH}{\mathcal{H}}
\newcommand{\cI}{\mathcal{I}}
\newcommand{\cB}{\mathcal{B}}

\newcommand{\cF}{\mathcal{F}}
\newcommand{\cW}{\mathcal{W}}

\newcommand{\cL}{\mathcal{L}}

\theoremstyle{plain}

\theoremstyle{theorem}

\newtheorem{defn}{Definition}[section]
\newtheorem{prop}[defn]{Proposition}
\newtheorem{thm}[defn]{Theorem}

\newtheorem{conj}[defn]{Conjecture}

\newtheorem{rmk}[defn]{Remark}

\theoremstyle{remark}

\makeatletter
\@namedef{subjclassname@2020}{\textup{2020} Mathematics Subject Classification}
\makeatother

\begin{document}
		
		\title[PRIMALITY OF WEAKLY CONNECTED COLLECTIONS OF CELLS...]{PRIMALITY OF WEAKLY CONNECTED COLLECTIONS OF CELLS AND WEAKLY CLOSED PATH POLYOMINOES}
		
		\author{CARMELO CISTO}
		\address{Universit\'{a} di Messina, Dipartimento di Scienze Matematiche e Informatiche, Scienze Fisiche e Scienze della Terra\\
			Viale Ferdinando Stagno D'Alcontres 31\\
			98166 Messina, Italy}
		\email{carmelo.cisto@unime.it}

		\author{FRANCESCO NAVARRA}
		\address{Universit\'{a} di Messina, Dipartimento di Scienze Matematiche e Informatiche, Scienze Fisiche e Scienze della Terra\\
			Viale Ferdinando Stagno D'Alcontres 31\\
			98166 Messina, Italy}
		\email{francesco.navarra@unime.it}

	\author{ROSANNA UTANO}
	\address{Universit\'{a} di Messina, Dipartimento di Scienze Matematiche e Informatiche, Scienze Fisiche e Scienze della Terra\\
	Viale Ferdinando Stagno D'Alcontres 31\\
	98166 Messina, Italy}
	\email{rosanna.utano@unime.it}

		

		\begin{abstract}
	In this paper we study the primality of weakly connected collections of cells, showing that the ideal generated by inner 2-minors attached to a weakly connected and simple collection of cells is the toric ideal of the edge ring of a weakly chordal bipartite graph. As an application of this result we characterize the primality of the polyomino ideals of weakly closed paths, a new class of non simple polyominoes.
		\end{abstract}

		\maketitle
		
	\section{Introduction}
	
	\noindent The study of the ideals of $t$-minors of an $m \times n$ matrix of indeterminates is a central topic in Commutative Algebra. The determinantal ideals are studied in \cite{conca1}, \cite{conca2} and \cite{conca3}, the ideals of adjacent 2-minors are studied in \cite{adiajent1},\cite{adiajent3} and \cite{adjent 2} as well as the ideals generated by an arbitrary set
	of $2$-minors in a $2\times n$ matrix in \cite{2.n}. The class of polyomino ideals generalizes the class of the ideals generated by 2-minors of $m\times n$ matrices of indeterminates. Roughly speaking, polyominoes are finite collections of squares of the same size joined edge to edge. In \cite{Qureshi} A.A. Qureshi establishes a connection of polyominoes to Commutative Algebra, attaching to a polyomino $\cP$ the ideal generated by all inner 2-minors of $\cP$. This ideal is called the \textit{polyomino ideal} of $\cP$ and it is denoted by $I_{\cP}$. We say that a polyomino $\cP$ is prime if the polyomino ideal $I_{\cP}$ is prime. One of the most exciting challenges is to characterize the primality of $I_{\cP}$ depending upon the shape of $\cP$. In \cite{Simple equivalent balanced} and in \cite{Simple are prime} it is proved that the simple polyominoes, which are roughly speaking the polyominoes without holes, are prime. 
	Nowadays the study is applied to multiply connected polyominoes, which are polyominoes with one or more holes. In \cite{Not simple with localization} and \cite{Shikama}, the authors prove that the polyominoes obtained by removing a convex polyomino from a rectangle in $\NN^2$ are prime. In \cite{Trento} the authors introduce a very useful tool to provide a characterization of prime polyominoes. They  define a particular sequence of inner intervals of $\cP$, called a \textit{zig-zag walk}, and they prove that if $I_{\cP}$ is prime then $\cP$ does not contain zig-zag walks. Using a computational method, they show that for polyominoes consisting of at most fourteen cells the non-existence of zig-zag walks in $\cP$ is a sufficient condition in order for $I_{\cP}$ to be prime. Therefore, they conjecture that non-existence of zig-zag walks in a polyomino characterizes its primality. In order to support this conjecture, in \cite{Cisto_Navarra} the authors introduce a new class of polyominoes, called closed paths, and they prove that having no zig-zag walks is a necessary and sufficient condition for their primality. Moreover they describe some classes of prime polyominoes, which can be considered as a generalization of closed paths. In \cite{Trento2} the authors give some conditions so that the set of generators of $I_{\cP}$ forms a reduced Gr\"obner basis with respect to some suitable degree reverse lexicographic monomial orders and they have proved that in these cases the polyomino is prime. Using this method, they prove the primality of two classes of thin polyominoes, which are polyominoes not containing the square tetromino. In addition, in  \cite{Trento3} the Hilbert series of simple thin polyominoes is studied.\\ 
	In this paper we study the primality of the weakly connected collections of cells and of a new class of non-simple polyominoes, called \textit{weakly closed paths}. In Section \ref{Section: Introduction} we introduce the preliminary notions and some useful tools. Inspired by \cite{Simple are prime}, in Section \ref{Section: Bipartite graph and edge ring} we define a bipartite graph $G(\cP)$ attached to a weakly connected and simple collection $\cP$ of cells of $\ZZ^2$, we prove that $G(\cP)$ is weakly chordal and that the ideal generated by inner 2-minors of $\cP$ is the toric ideal attached to the edge ring of $G(\cP)$. This result generalizes Theorem 3.10 of \cite{Simple are prime} and we conjecture that absence of zig-zag walks in a weakly connected collection of cells characterizes its primality. Applying the previous result and using the same demonstrative approach introduced in \cite{Cisto_Navarra}, in Section \ref{Section: Weakly closed paths} we obtain that the new class of multiply connected polyominoes, called \textit{weakly closed paths}, satisfies the conjecture, which states that a polyomino ideal is prime if and only if the polyomino does not contain zig-zag walks.

	\section{Preliminaries on polyominoes and polyomino ideals}\label{Section: Introduction}
	 
\noindent Let $(i,j),(k,l)\in \Z^2$, we define the natural partial order on $\Z^2$ as $(i,j)\leq(k,l)$ if $i\leq k$ and $j\leq l$. Let $a=(i,j),b=(k,l)\in\Z^2$ with $a\leq b$. The set $[a,b]=\{(r,s)\in \Z^2: i\leq r\leq k,\ j\leq s\leq l \}$ is called an \textit{interval} of $\Z^2$. 
If $i< k$ and $j< l$, then $[a,b]$ is a \textit{proper} interval. The elements $a, b$ are the \textit{diagonal corners} and $c=(i,l)$, $d=(k,j)$ the \textit{anti-diagonal corners} of $[a,b]$. If $j=l$ (or $i=k$) we say that $a$ and $b$ are in \textit{horizontal} (or \textit{vertical}) \textit{position}. An interval as $C=[a,a+(1,1)]$ is called a \textit{cell} of $\ZZ^2$ and $a$ is the \textit{lower left corner} of $C$. The elements $a$, $a+(0,1)$, $a+(1,0)$ and $a+(1,1)$ are called the \textit{vertices} or \textit{corners} of $C$ and the sets $\{a,a+(1,0)\}$, $\{a+(1,0),a+(1,1)\}$, $\{a+(0,1),a+(1,1)\}$ and $\{a,a+(0,1)\}$ are called the \textit{edges} of $C$. We denote the set of the vertices and the edges of $C$ respectively by $V(C)$ and $E(C)$. If $C$ is a cell of $\Z^2$, then we say that a cell $D$ of $\Z^2$ is adjacent to $C$ if $C\cap D$ is a common edge of $C$ and $D$.\\ 
Let $\cP$ be a non-empty collection of cells in $\Z^2$. The set of the vertices and the edges of $\cP$ are respectively $V(\cP)=\bigcup_{C\in \cP}V(C)$ and $E(\cP)=\bigcup_{C\in \cP}E(C)$. Let $C$ and $D$ be two distinct cells of $\cP$. A \textit{walk} from $C$ to $D$ is a sequence $\cC:C=C_1,\dots,C_m=D$ of cells of $\ZZ^2$ such that $C_i \cap C_{i+1}$ is an edge of $C_i$ and $C_{i+1}$ for $i=1,\dots,m-1$. If in addition $C_i \neq C_j$ for all $i\neq j$, then $\cC$ is called a \textit{path} from $C$ to $D$. We say that $C$ and $D$ are connected if there exists a path of cells in $\cP$ from $C$ to $D$. 
Let $\cP$ be a non-empty, finite collection of cells in $\Z^2$. We say that $\cP$ is a \textit{polyomino} if any two cells of $\cP$ are connected. For instance, see Figure \ref{Figura: A polyomino+weakly connected} (A). We say that $\cP$ is \textit{weakly connected} if for any two cells $C$ and $D$ in $\cP$ there exists a sequence of cells $\cC: C=C_1,\dots,C_m=D$ of $\cP$ such that $V(C_i)\cap V(C_{i+1}) \neq \emptyset$  for all $i=1,\dots,m-1$. Let $\cP'$ be a subset of cells of $\cP$. $\cP'$ is called a \textit{connected component} of $\cP$ if $\cP'$ is a polyomino and it is maximal with respect to the set inclusion, that is if $A\in \cP\setminus \cP'$ then $\cP'\cup \{A\}$ is not a polyomino.	For instance, see Figure \ref{Figura: A polyomino+weakly connected} (B).
\begin{figure}[h!]
	\centering
	\subfloat[]{\includegraphics[scale=0.5]{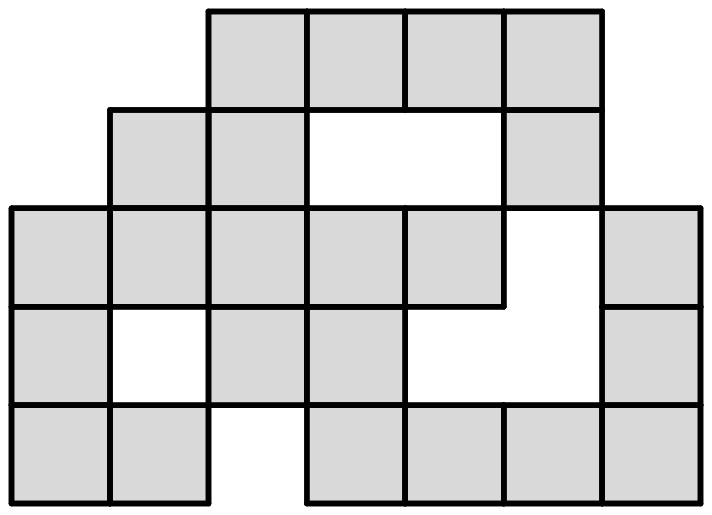}}\qquad\qquad\qquad
	\subfloat[]{\includegraphics[scale=0.5]{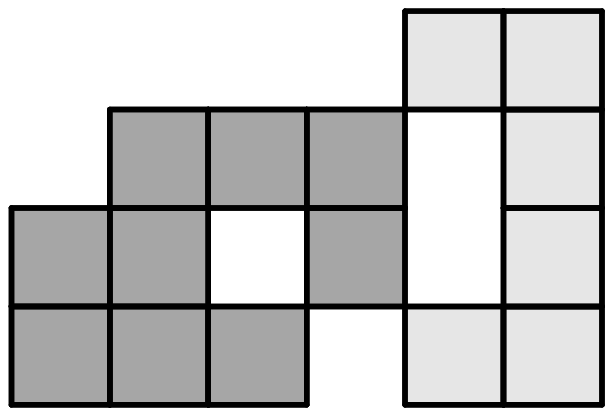}}
	\caption{}
	\label{Figura: A polyomino+weakly connected}
\end{figure}\\
	 We say that $\cP$ is \textit{simple} if for any two cells $C$ and $D$ of $\Z^2$, not in $\cP$, there exists a path $\cC: C=C_1,\dots,C_m=D$ such that $C_i\notin \cP$  for all $i=1,\dots,m$. For example, the polyomino and the weakly connected collection of cells in Figure \ref{Figura: A polyomino+weakly connected} are not simple. A finite collection of cells $\cH$ not in $\cP$ is a \textit{hole} of $\cP$ if any two cells $F$ and $G$ of $\cH$ are connected by a path $\cF:F=F_1,\dots,F_t=G$ such that $F_j\in \cH$ for all $j=1,\dots,t$ and $\cH$ is maximal with respect to set inclusion. Observe that each hole of a collection $\cP$ of cells is a simple polyomino and $\cP$ is simple if and only if it has not any hole.\\
	 Let $A$ and $B$ be two cells of $\Z^2$ and let $a=(i,j)$ and $b=(k,l)$ be the lower left corners of $A$ and $B$, with $a\leq b$. The \textit{cell interval}, denoted by $[A,B]$, is defined by the set of the cells of $\Z^2$ with lower left corner $(r,s)$ such that $i\leqslant r\leqslant k$ and $j\leqslant s\leqslant l$. If $(i,j)$ and $(k,l)$ are in horizontal position, we say that the cells $A$ and $B$ are in horizontal position. Similarly, we define two cells in vertical position. Let $A$ and $B$ be two cells of $\cP$ in vertical or horizontal position. The cell interval $[A,B]$ is called a
	 \textit{block of $\cP$ of length n} if any cell $C$ of $[A,B]$ belongs to $\cP$ and $|[A,B]|=n$. The cells $A$ and $B$ are called \textit{extremal} cells of $[A,B]$. The block $[A,B]$ is \textit{maximal} if there does not exist any block $[A',B']$ of $\cP$ such that $[A,B]\subset [A',B']$. Moreover if $A$ and $B$ are in vertical (resp. horizontal) position, then $[A,B]$ is also called a \textit{maximal vertical (resp. horizontal) block} of $\cP$. 
	An interval $[a,b]$ with $a=(i,j)$, $b=(k,j)$ and $i<k$ is called a \textit{horizontal edge interval} of $\cP$ if the sets $\{(\ell,j),(\ell+1,j)\}$ are edges of cells of $\cP$ for all $\ell=i,\dots,k-1$. In addition, if $\{(i-1,j),(i,j)\}$ and $\{(k,j),(k+1,j)\}$ do not belong to $E(\cP)$, then $[a,b]$ is called a \textit{maximal} horizontal edge interval of $\cP$. We define similarly a \textit{vertical edge interval} and a \textit{maximal} vertical edge interval. Observe that a lattice interval of $\ZZ^2$ identifies a cell interval of $\ZZ^2$ and vice versa, so if $I$ is an interval of $\ZZ^2$ we denote by $\cP_I$ the cell interval associated to $I$.
	A proper interval $[a,b]$ is called an \textit{inner interval} of $\cP$ if all cells of $[a,b]$ belong to $\cP$.\\
	Following \cite{Trento} we recall the definition of a \textit{zig-zag walk} of $\cP$. A zig-zag walk of $\cP$ is a sequence $\cW:I_1,\dots,I_\ell$ of distinct inner intervals of $\cP$ where, for all $i=1,\dots,\ell$, the interval $I_i$ has either diagonal corners $v_i$, $z_i$ and anti-diagonal corners $u_i$, $v_{i+1}$ or anti-diagonal corners $v_i$, $z_i$ and diagonal corners $u_i$, $v_{i+1}$, such that:
	\begin{enumerate}
		\item $I_1\cap I_\ell=\{v_1=v_{\ell+1}\}$ and $I_i\cap I_{i+1}=\{v_{i+1}\}$, for all $i=1,\dots,\ell-1$;
		\item $v_i$ and $v_{i+1}$ are on the same edge interval of $\cP$, for all $i=1,\dots,\ell$;
		\item for all $i,j\in \{1,\dots,\ell\}$ with $i\neq j$, there exists no inner interval $J$ of $\cP$ such that $z_i$, $z_j$ belong to $V(J)$.
	\end{enumerate}
	Let $\cP$ be a non-empty finite collection of
	cells in $\Z^2$. Let $K$ be a field and $S=K[x_v\mid v\in V(\cP)]$. Consider a proper interval $[a,b]$ of $\ZZ^2$, with $a$,$b$ diagonal corners and $c$,$d$ anti-diagonal ones. We attach the binomial $x_ax_b-x_cx_d$ to $[a,b]$ and if $[a,b]$ is an inner interval then the binomial $x_ax_b-x_cx_d$ is called an \textit{inner 2-minor} of $\cP$. We denote by $I_{\cP}\subset S$ the ideal in $S$ generated by all the inner 2-minors of $\cP$. We set also $K[\cP] = S/I_{\cP}$, that is the coordinate ring of $\cP$. If $\cP$ is a polyomino, the ideal $I_{\cP}$ is called the \textit{polyomino ideal of $\cP$}.\\

	\section{Bipartite graph and edge ring associated to a simple collection of cells}\label{Section: Bipartite graph and edge ring}
\noindent Let $\cP$ be a weakly connected collection of cells of $\ZZ^2$. Let $\{V_i\}_{i\in I}$ be the sets of the maximal vertical edge intervals of $\cP$ and $\{H_j\}_{j\in J}$ be the set of the maximal horizontal edge intervals of $\cP$. Let $\{v_i\}_{i\in I}$ and $\{h_j\}_{j\in J}$ be two sets of variables associated respectively to $\{V_i\}_{i\in I}$ and $\{H_j\}_{j\in J}$. We associate to $\cP$ a bipartite graph $G(\cP)$, whose vertex set is $V(G(\cP))=\{v_i\}_{i\in I}\sqcup \{h_j\}_{j\in J}$ and edge set is $E(G(\cP))=\big \{ \{v_i,h_j\} | V_i\cap H_j\in V(\cP)\big\}$. For instance, Figure \ref{Figura_colezione_di_celle_con_grafo} illustrates a collection of cells $\cP$ on the left and its associated bipartite graph $G(\cP)$ on the right.
	
		\begin{figure}[h]
		\centering
		\subfloat{\includegraphics[scale=0.8]{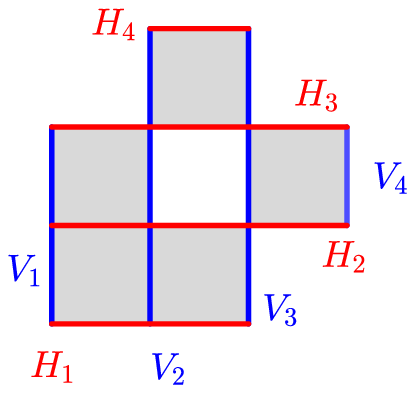}}\qquad\qquad
		\subfloat{\includegraphics[scale=0.7]{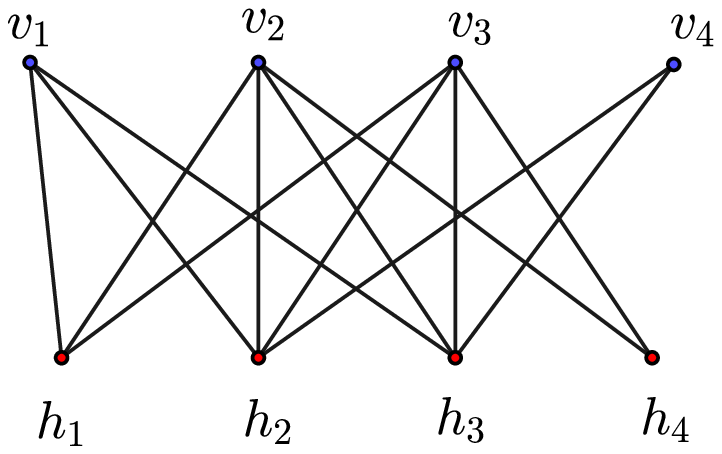}}
		\caption{}
		\label{Figura_colezione_di_celle_con_grafo}
	\end{figure}

\noindent In the bipartite graph $G(\cP)$ a cycle $\cC_{G(\cP)}$ of length $2r$ is a subset $\{v_{i_1},h_{j_1},\dots,v_{i_{r-1}},h_{j_{r-1}},v_{i_r},h_{j_r}\}$ of distinct vertices of $V(G(\cP))$ such that $\{v_{i_k},h_{j_k}\}$ and $\{h_{j_k},v_{i_{k+1}}\}$ belong to $E(G(\cP))$ for all $k=1,\dots,r$, where $i_{r+1}=i_1$. Since $\{v_{i_k},h_{j_k}\}\in E(G(\cP))$, $V_{i_k}\cap H_{j_k}$ is a vertex of $\cP$ for all $k=1,\dots,r$; similarly, since $\{h_{j_k},v_{i_{k+1}}\}\in E(G(\cP))$, $V_{i_{k+1}}\cap H_{j_k}$ is a vertex of $\cP$ for all $k=1,\dots,r$, where $i_{r+1}=i_1$. We can associate to each cycle $\cC_{G(\cP)}$ in $G(\cP)$ the following binomial:
\[
		f_{\cC_{G(\cP)}}=x_{V_{i_1}\cap H_{j_1}}\dots x_{V_{i_r}\cap H_{j_r}}-x_{V_{i_2}\cap H_{j_1}}\dots x_{V_{i_1}\cap H_{j_r}}
\]
  Following \cite{def balanced}, we recall the definition of a cycle in $\cP$. A cycle $\cC_{\cP}$ in $\cP$ is a sequence $a_1,\dots,a_m$ of vertices of $\cP$ such that:
\begin{enumerate}
	\item $a_1=a_m$;
	\item $a_i\neq a_j$ for all $i\neq j$ with $i,j\in\{1,\dots,m-1\}$;
	\item $[a_i,a_{i+1}]$ is a horizontal or vertical edge interval of $\cP$ for all $i=1,\dots,m-1$;
	\item for all $i=1,\dots,m$, if $[a_i,a_{i+1}]$ is a horizontal edge interval of $\cP$, then $[a_{i+1},a_{i+2}]$ is a vertical edge interval of $\cP$ and vice versa, with $a_{m+1}=a_2$.
\end{enumerate} 
The vertices $a_1\dots,a_{m-1}$ of $\cP$ are called \textit{vertices of $\cC_{\cP}$} and we set $V(\cC_{\cP})=\{a_1,\dots,a_{m-1}\}$. It follows from the definition of a cycle that $m$ is odd, so we can consider the following binomial
\[
	f_{\cC_{\cP}}=\prod_{k=1}^{(m-1)/2}x_{a_{2k-1}}-\prod_{k=1}^{(m-1)/2}x_{a_{2k}}
\]
and we can attach to each cycle $\cC_{\cP}$ in $\cP$  the binomial $f_{\cC_{\cP}}$. Moreover, a cycle in $\cP$ is called \textit{primitive} if each maximal edge interval of $\cP$ contains at most two vertices of $\cC_{\cP}$.

\begin{rmk}\rm \label{Remark_Ad un ciclo in G(P) corrisponde ciclo in P}

Arguing as in Section 1 of \cite{Simple are prime}, a cycle $\cC_{G(\cP)}=\{v_{i_1},h_{j_1},v_{i_2},h_{j_2}\dots,v_{i_{r-1}},h_{j_{r-1}},v_{i_r},h_{j_r}\}$ of the bipartite graph $G(\cP)$ associated to $\cP$ defines a primitive cycle $\cC_{\cP}: V_{i_1}\cap H_{j_1}, V_{i_2}\cap H_{j_1}, V_{i_2}\cap H_{j_2},\dots, V_{i_r}\cap H_{j_r}, V_{i_1}\cap H_{j_r},  V_{i_1}\cap H_{j_1}$ in $\cP$  and vice versa. Moreover, we have also $f_{\cC_{G(\cP)}}=f_{\cC_{\cP}}$. 
\end{rmk}
\noindent Recall that a graph is called \textit{weakly chordal} if every cycle of length greater than 4 has a chord. According to \cite{Simple are prime}, if $\cC_{\cP}:a_1,\dots, a_m$ is a cycle in $\cP$ then $\cC_{\cP}$ has a \textit{self-crossing} if there exist two indices $i,j\in \{1,\dots,m-1\}$ such that:
\begin{enumerate}
\item $a_i,a_{i+1}\in V_k$ and $a_j,a_{j+1}\in H_l$ for some $k\in I$ and $l\in J$;
\item $a_i,a_{i+1},a_{j},a_{j+1}$ are all distinct;
\item $V_k\cap H_l\neq \emptyset$.
\end{enumerate} 
In such a case, as in Section 2 of \cite{Simple are prime}, if $\cC_{\cP}$ is a primitive cycle in $\cP$ having a self-crossing, then $\cC_{G(\cP)}$ has a chord.

\noindent 
Moreover in \cite{Simple are prime} the authors show that the polyomino ideal attached to a simple polyomino is the toric ideal of the edge ring of the weakly chordal graph $G(\cP)$. Now, we give a generalization of these results, which will be useful and crucial in the last section.
\begin{prop}\label{Prop: Se P è debolmente connessa e semplice, allora G(P) è debolm cordale}
	Let $\cP$ be a weakly connected and simple collection of cells. Then $G(\cP)$ is weakly chordal.
\end{prop}
\begin{proof}
	We may assume that $\cP$ has two connected components, denoted by $\cP_1$ and $\cP_2$. The arguments are similar if $\cP$ has more than two connected components. Let $V(\cP_1)\cap V(\cP_2)=\{\tilde{v}\}$. Let $\cC_{G(\cP)}=\{v_{i_1},h_{j_1},\dots,v_{i_r},h_{j_r}\}$ be a cycle of $G(\cP)$ of length $2r$ with $r\geq 3$. By Remark \ref{Remark_Ad un ciclo in G(P) corrisponde ciclo in P} we obtain that $\cC_{G(\cP)}$ defines a primitive cycle in $\cP$ 
	 \[
	\cC_{\cP}: V_{i_1}\cap H_{j_1}, V_{i_2}\cap H_{j_1}, V_{i_2}\cap H_{j_2},\dots, V_{i_r}\cap H_{j_r}, V_{i_1}\cap H_{j_r},  V_{i_1}\cap H_{j_1}
	\]
	We set $a_1=V_{i_1}\cap H_{j_1},a_2= V_{i_2}\cap H_{j_1},\dots, a_{2r-1}=V_{i_r}\cap H_{j_r}, a_{2r}=V_{i_1}\cap H_{j_r},a_{2r+1} = V_{i_1}\cap H_{j_1}$.
	We distinguish two different cases. Firstly, we suppose that all vertices of $\cC_{\cP}$ are either in $V(\cP_1)$ or $V(\cP_2)$. We may assume that $a_k\in V(\cP_1)$ for all $k=1,\dots,2r$. We prove that $\cC_{G(\cP)}$ has a chord. 
	Observe that $\cP_1$ is a simple polyomino, otherwise $\cP$ is not a simple collection of cells. Consider the bipartite graph $G(\cP_1)$ attached to $\cP_1$. By Lemma 2.1 in \cite{Simple are prime} it follows that $G(\cP_1)$ is weakly chordal, hence the cycle $\cC_{G(\cP)}$ has a chord. \\ 
	In the second case, we suppose that there exist two distinct vertices different from $\tilde{v}$, one belonging to $V(\cP_1)$ and the other to $V(\cP_2)$.  We prove that $\cC_{\cP}$ has a self-crossing. We denote by $V_{\tilde{v}}$ and $H_{\tilde{v}}$ respectively the vertical and horizontal maximal edge intervals of $\cP$, such that $V_{\tilde{v}}\cap H_{\tilde{v}}=\{\tilde{v}\}$. It is not restrictive to assume that $a_1\in V(\cP_1)\backslash\{\tilde{v}\}$. Let $i$ be the smallest integer such that $a_i\in V(\cP_1)$ and $a_{i+1}\in V(\cP_2)$. We can assume that $[a_i,a_{i+1}]$ is a horizontal interval of $\cP$, so $[a_i,a_{i+1}]$ is contained in $H_{\tilde{v}}$ and it is obvious that $\tilde{v}\in [a_i,a_{i+1}]$.  We note that $a_{2r+1} =a_1\in V(\cP_1)\backslash\{\tilde{v}\}$. Then there exists $p\in\{i+2,\dots,2r\}$ such that $\tilde{v}\in[a_p,a_{p+1}]$, with $a_p,a_{p+1}\notin \{\tilde{v}\}$. Moreover, we note that from the primitivity of $\mathcal{C}_\cP$ it follows immediately that $[a_p,a_{p+1}]\subseteq V_{\tilde{v}}$. Hence we obtain that there exist two distinct indices $i,p\in\{1,\dots,2r\}$ such that:
	\begin{enumerate}
		\item $a_i,a_{i+1}\in H_{\tilde{v}}$ and $a_p,a_{p+1}\in V_{\tilde{v}}$;
		\item $a_i,a_{i+1},a_p,a_{p+1}$ are all distinct because they are the vertices of a primitive cycle in $\cP$;
		\item $V_{\tilde{v}}\cap H_{\tilde{v}}\neq \emptyset$ because obviously $V_{\tilde{v}}\cap H_{\tilde{v}}=\{\tilde{v}\}$.
	\end{enumerate}
 In conclusion, $\cC_{\cP}$ has a self-crossing and as a consequence $\cC_{G(\cP)}$ has a chord.\end{proof} 

\noindent We define the following map:
\begin{align*}
\alpha: V(\cP)&\longrightarrow K[\{v_i,h_j\}:i\in I,j\in J]\\
r&\longmapsto  v_ih_j
\end{align*}
with $r\in V_i\cap H_j$. The toric ring $K[\alpha(v):v\in V(\cP)]$ can be viewed as the edge ring of $G(\cP)$ and it is denoted by $K[G(\cP)]$. Let $S$ be the polynomial ring $K[x_r:r\in V(\cP)]$ and let us consider the following surjective ring homomorphism:
\begin{align*}
\phi: S &\longrightarrow K[G(\cP)]\\
\phi(x_r&)=\alpha(r)
\end{align*}
The toric ideal $J_{\cP}$ is the kernel of $\phi$. It is known (\cite{Ohshugi-Hibi_Koszul graph}, \cite{Ohshugi-Hibi_Generatori ideali torici}) that if the bipartite graph $G(\cP)$ is weakly chordal then the associated toric ideal $J_{\cP}$ is minimally generated by quadratic binomials attached to the cycles of $G(\cP)$ of length 4. 
\begin{thm}\label{Teorema: P collezione semplice dedebolmente connessa allora I=J}
	Let $\cP$ be a weakly connected and simple collection of cells. Then $I_{\cP}=J_{\cP}$.
\end{thm}
\begin{proof}
	Assume that $\cP$ consists of the connected components $\cP_1,\dots,\cP_m$, with $m\geq1$.
	We prove firstly that $I_{\cP}\subseteq J_\cP$. Let $f$ be a generator of $I_{\cP}$, so there exists an inner interval $[a,b]$ of $\cP$, such that $f=x_ax_b-x_cx_d$, where $c,d$ are the anti-diagonal corners of $[a,b]$. It is clear that $\phi(x_ax_b)=\alpha(a)\alpha(b)=\alpha(c)\alpha(d)=\phi(x_cx_d)$, so $f\in J_{\cP}$. Therefore $I_{\cP}\subseteq J_\cP$.  
	We prove that $J_{\cP}\subseteq I_\cP$. By Proposition \ref{Prop: Se P è debolmente connessa e semplice, allora G(P) è debolm cordale} the bipartite graph $G(\cP)$ attached to $\cP$ is weakly chordal, so $J_{\cP}$ is generated minimally by quadratic binomials attached to cycles of $G(\cP)$ of length 4. Let $f$ be a generator of $J_{\cP}$. Then there exists a cycle of $G(\cP)$ of length 4, $\cC_{G(\cP)}:v_{i_1},h_{j_1},v_{i_2},h_{j_2}$, such that $f=f_{\cC_{G(\cP)}}$. By Remark \ref{Remark_Ad un ciclo in G(P) corrisponde ciclo in P} $\cC_{G(\cP)}$ defines the following primitive cycle in $\cP$: 
	\[
	\cC_{\cP}: V_{i_1}\cap H_{j_1}, V_{i_2}\cap H_{j_1}, V_{i_2}\cap H_{j_2}, V_{i_1}\cap H_{j_2}, V_{i_1}\cap H_{j_1}.\]
	We set $a_1=V_{i_1}\cap H_{j_1},a_2= V_{i_2}\cap H_{j_1}, a_{3}=V_{i_2}\cap H_{j_2}, a_{4}=V_{i_1}\cap H_{j_2}$ and we have $f=f_{\cC_{G(\cP)}}= f_{\cC_{\cP}}$. 
	Since $\cP$ is a simple collection of cells and $f=f_{\cC_{\cP}}$, there exists $j\in\{1,\dots,m\}$ such that $a_i\in V(\cP_j)$ for all $i=1,2,3,4$.
	 Consider the map $\phi'$ as the restriction of $\phi$ to $K[x_a:a\in V(\cP_j)]$ and we denote by $J_{\cP_j}$ the kernel of $\phi'$. By Theorem 2.2 in \cite{Simple are prime} it follows that $J_{\cP_j}=I_{\cP_j}$, where $I_{\cP_j}$ is the polyomino ideal associated to $\cP_j$. Hence we have $f\in J_{\cP_j}=I_{\cP_j}\subseteq I_{\cP}$. Therefore $J_{\cP}\subseteq I_{\cP}$.
\end{proof}

\begin{rmk}\rm
We observe that there exist weakly connected and non-simple collections of cells that are not prime. The collection of cells in Figure \ref{Figura:collezione di celle non semplici e non prima} (A) is non-simple and weakly connected with four connected components but it is not prime. Its non-primality follows by \cite[case (2) of Theorem 3.2, Corollary 3.6]{Qureshi}. Conversely, in Figure \ref{Figura:collezione di celle non semplici e non prima} (B) there is a weakly connected and non-simple collection of cells which is prime. For the proof of its primality we refer to Remark \ref{Remark: dopo la dim coll non sempl e deb connessa primo} of Section \ref{Section: Weakly closed paths}.
\begin{figure}[h!]
	\centering
	\subfloat[]{\includegraphics[scale=0.7]{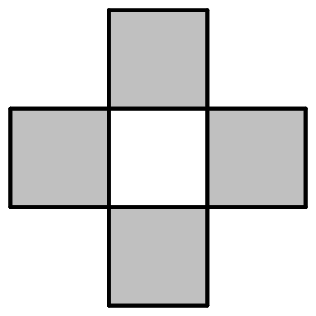}}\qquad \qquad\qquad
	\subfloat[]{\includegraphics[scale=0.7]{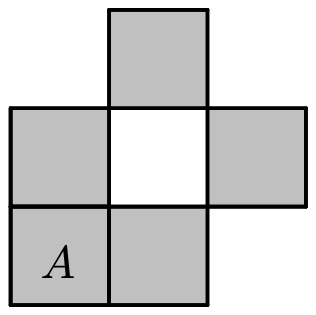}}
	\caption{}
	\label{Figura:collezione di celle non semplici e non prima}
\end{figure}
\end{rmk}

\noindent In according to previous arguments it is natural to generalize the conjecture given in \cite{Trento} for weakly connected collections of cells.

\begin{conj}
	Let $\cP$ be a weakly connected collection of cells. The following are equivalent:
	\begin{enumerate}
		\item $I_{\cP}$ is prime;
		\item $\cP$ has no zig-zag walks.
	\end{enumerate}
\end{conj}

\section{Weakly closed path polyominoes and their primality} \label{Section: Weakly closed paths}

\noindent In this section we introduce a new class of polyominoes, which we call weakly closed path polyominoes. As an application of Theorem \ref{Teorema: P collezione semplice dedebolmente connessa allora I=J} and by using similar techniques of \cite{Cisto_Navarra}, we characterize all weakly closed paths having no zig-zag walks and their primality.
\begin{defn}\label{Definizione: weakly closed path}\rm 
	A finite non-empty collection of cells $\cP$ is called a \textit{weakly closed path} if it is a path of $n$ cells $A_1,\dots,A_{n-1},A_n=A_0$ with $n>6$ such that:
	\begin{enumerate}
		\item $|V(A_0)\cap V(A_1)|=1$;
		\item $V(A_2)\cap V(A_{0})=V(A_{n-1})\cap V(A_1)=\emptyset$;
		\item $V(A_i)\cap V(A_j)=\emptyset$ for all $i\in\{1,\dots,n\}$ and for all $j\notin\{i-2,i-1,i,i+1,i+2\}$, where the indices are reduced modulo $n$.
	\end{enumerate}
\end{defn}
\noindent We call the unique vertex $v_H$ in $V(A_0)\cap V(A_1)$ a \textit{hooking corner}. Note that a weakly closed path is a non-simple polyomino having a unique hole. In Figure \ref{Figura:esempi di weakly closed path} there are some examples of weakly closed paths. \\
The difference between a closed path, introduced in \cite{Cisto_Navarra}, and a weakly closed path is subtle but quite deep. In fact in a closed path it is possible to order the cells in such a way that every cell has an edge in common with its consecutive cell. In a weakly closed path the same holds, with the exception of exactly two consecutive cells that have just a vertex in common. These polyominoes, as well as the closed paths, are particular thin polyominoes, which are polyominoes not containing the square tetromino.
\begin{figure}[h]
	\centering
	\subfloat[]{\includegraphics[scale=0.55]{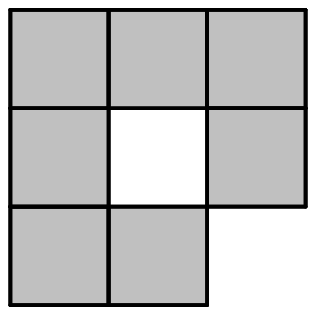}}\qquad \qquad
	\subfloat[]{\includegraphics[scale=0.45]{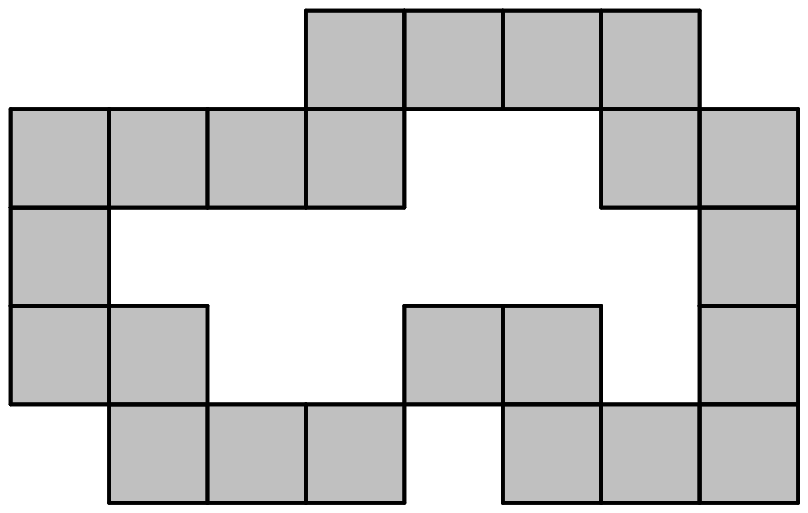}} \qquad \qquad
	\subfloat[]{\includegraphics[scale=0.45]{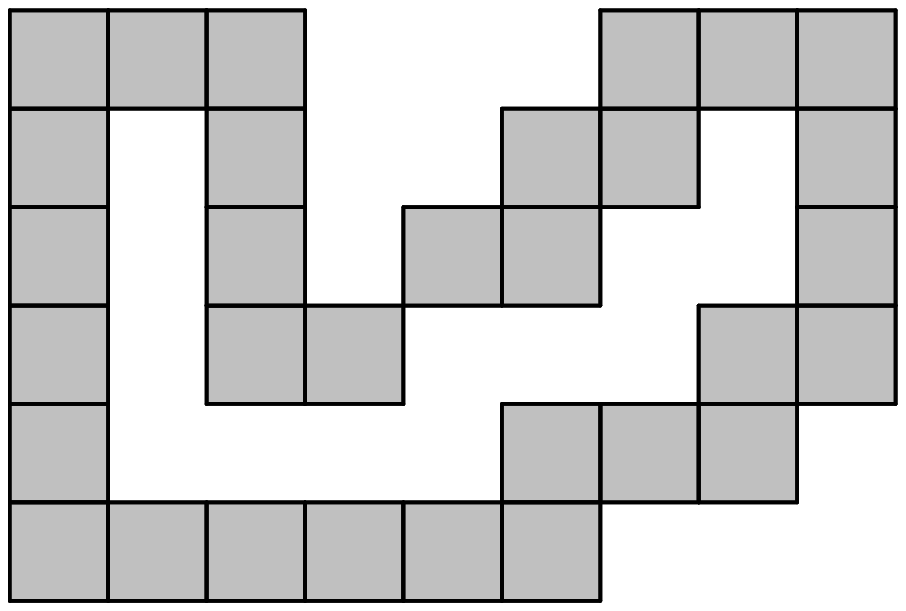}}
	\caption{Examples of weakly closed path polyominoes.}
	\label{Figura:esempi di weakly closed path}
\end{figure}

\noindent  Let $\cP$ be a polyomino. A \textit{weak L-configuration} is a finite collection of cells of $\cP$ such that:
\begin{enumerate}
	\item it consists of a maximal horizontal (resp. vertical) block $[A,B]$ of length two, a vertical (resp. horizontal) block $[D,F]$ of length at least two and a cell $C$ of $\cP$, not belonging to $[A,B]\sqcup [D,F]$;
	\item $V(C)\cap V([A,B])=\{a_1\}$ and $V([D,F])\cap V([A,B])=\{a_2,b_2\}$, where $a_2\neq b_2$;
	\item $[a_2,b_2]$ is on the same maximal horizontal (resp. vertical) edge interval of $\cP$ containing $a_1$ (see Figure \ref{Figura:esempio weak L configuration}). 
\end{enumerate}

\begin{figure}[h]
	\centering
	\subfloat{\includegraphics[scale=0.7]{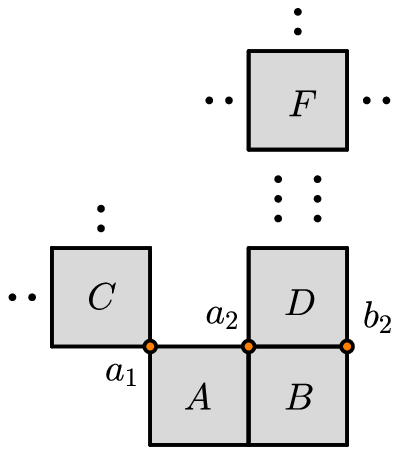}}\qquad\qquad
	\subfloat{\includegraphics[scale=0.7]{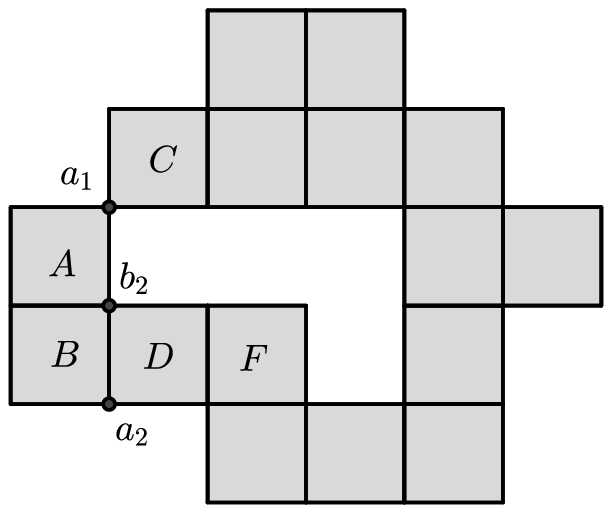}} 
	\caption{A Weak L-configuration and a polyomino
containing a weak L-configuration}
	\label{Figura:esempio weak L configuration}
\end{figure}

\noindent A finite collection of cells of $\cP$, made up of a maximal horizontal (resp. vertical) block $[A,B]$ of $\cP$ of length at least two and two distinct cells $C$ and $D$ of $\cP$, not belonging to $[A,B]$, with $V(C)\cap V([A,B])=\{a_1\}$ and $V(D)\cap V([A,B])=\{a_2,b_2\}$ where $a_2\neq b_2$, is called a \textit{weak ladder} if $[a_2,b_2]$ is not on the same maximal horizontal (resp. vertical) edge interval of $\cP$ containing $a_1$ (see Figure \ref{Figura:esempio weak ladder}).
\begin{figure}[h]
	\centering
	\subfloat{\includegraphics[scale=0.7]{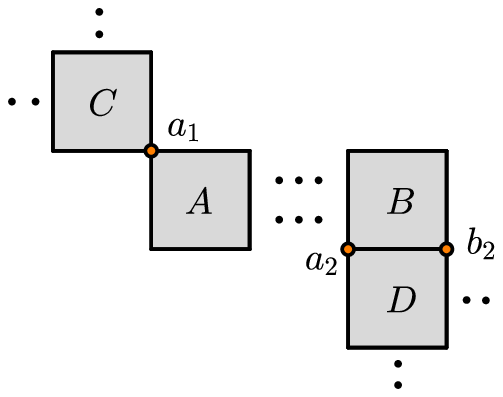}}\qquad\qquad
	\subfloat{\includegraphics[scale=0.7]{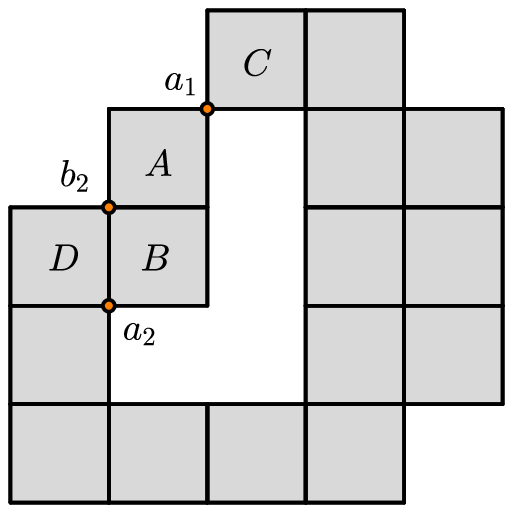}} 
	\caption{A Weak ladder and a polyomino containing a weak ladder}
	\label{Figura:esempio weak ladder}
\end{figure} 

\noindent  As introduced in \cite{Cisto_Navarra}, we say that a path of five cells $C_1, C_2, C_3, C_4, C_5$ of $\cP$ is an \textit{L-configuration} if the two sequences $C_1, C_2, C_3$ and $C_3, C_4, C_5$ go in two orthogonal directions. A set $\cB=\{\cB_i\}_{i=1,\dots,n}$ of maximal horizontal (or vertical) blocks of length at least two, with $V(\cB_i)\cap V(\cB_{i+1})=\{a_i,b_i\}$ and $a_i\neq b_i$ for all $i=1,\dots,n-1$, is called a \textit{ladder of $n$ steps} if $[a_i,b_i]$ is not on the same edge interval of $[a_{i+1},b_{i+1}]$ for all $i=1,\dots,n-2$. For instance, in Figure \ref{Figura:esempio L-configuration e ladder} we represent a polyomino with an $L$-configuration on the left and a polyomino having a ladder of three steps on the right.  
\begin{figure}[h]
	\centering
	\subfloat{\includegraphics[scale=0.7]{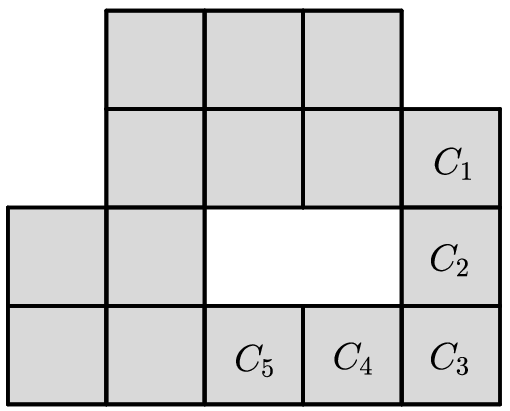}}\qquad\qquad
	\subfloat{\includegraphics[scale=0.7]{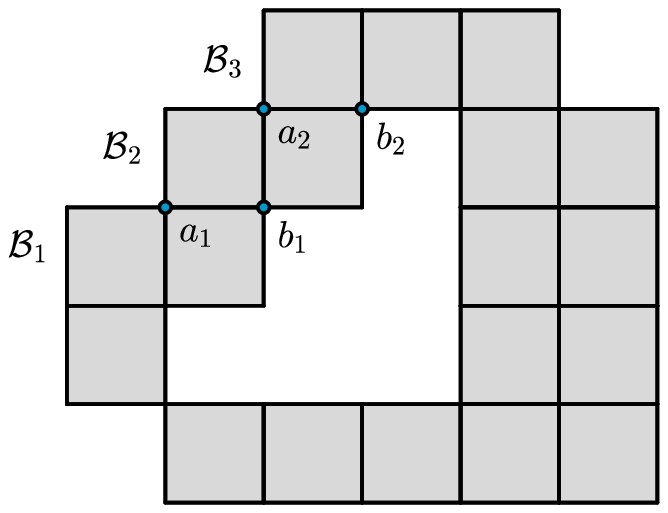}} 
	\caption{An example of $L$-configuration and horizontal ladder of three steps.}
	\label{Figura:esempio L-configuration e ladder}
\end{figure}

\begin{prop}\label{Proposizione: Se ha una L-conf, weak L-conf ecc, allora P non ha zig-zag}
	Let $\cP$ be a weakly closed path. If one of the following conditions holds:
	\begin{enumerate}
		\item $\cP$ has a weak $L$-configuration,
		\item $\cP$ has a weak ladder,
		\item $\cP$ has an $L$-configuration,
		\item $\cP$ has a ladder of at least three steps,
	\end{enumerate}
    then $\cP$ does not contain zig-zag walks.
\end{prop}
\begin{proof}
	(1) Suppose that $\cP$ has a weak $L$-configuration. Assume that the weak $L$-configuration is as in the picture on the left in Figure \ref{Figura:esempio weak L configuration}, otherwise we apply suitable reflections or rotations in order to have it in such a position.
	Suppose that there exists a sequence $\cW:I_1,\dots,I_\ell$ of distinct inner intervals of $\cP$, where for all $i=1,\dots,\ell$ the interval $I_i$ has diagonal corners $v_i$, $z_i$ and anti-diagonal corners $u_i$, $v_{i+1}$, such that $I_1\cap I_\ell=\{v_1=v_{\ell+1}\}$ and $I_i\cap I_{i+1}=\{v_{i+1}\}$, for all $i=1,\dots,\ell-1$. We may assume that $v_i$ and $v_{i+1}$ are on the same edge interval of $\cP$ for all $i=1,\dots,\ell$, otherwise we have finished. Observe that there exists $i\in\{1,\dots,\ell\}$ such that $C\in I_i$ and $I_i\cap I_r=\{a_1\}$ where $r=i-1$ or $r=i+1$. In such a case it is not restrictive to assume $1<i<\ell-1$ and $r=i+1$. It follows from the shape of the weak $L$-configuration that $I_{i+1}\cap I_{i+2}=\{a_{2}\}$ and $z_{i+1}$ is the lower right corner of $A$. The corner $z_{i+2}$ of $I_{i+2}$ belongs to $V([D,F])$ and is on the vertical edge interval of $\cP$ containing $b_2$, so $[B,F]$ is an inner interval such that $z_{i+1},z_{i+2}\in V([B,F])$. Therefore $\cP$ cannot contain zig-zag walks.  \\
	(2) Suppose that $\cP$ has a weak ladder. Assume that there exists a sequence $\cW:I_1,\dots,I_\ell$ of distinct inner intervals of $\cP$ such that $I_i\cap I_{i+1}=\{v_{i+1}\}$ for all $i=1,\dots,\ell-1$ and $I_1\cap I_\ell=\{v_1=v_{\ell+1}\}$. Then there exists $i\in\{1,\dots,n\}$ such that $C\in I_i$ and $I_i\cap I_r=\{a_1\}$ where $r=i-1$ or $r=i+1$. It is not restrictive to assume $1<i<\ell-1$ and $r=i+1$. For the shape of the weak ladder we have either $I_{i+1}\cap I_{i+2}=\{a_{2}\}$ or $I_{i+1}\cap I_{i+2}=\{b_{2}\}$. Then $v_{i+1}$ and $v_{i+2}$ are not on the same edge interval of $\cP$, so $\cP$ cannot contain zig-zag walks. \\
	(3) If $\cP$ has an $L$-configuration, then we have the desired conclusion by arguing as done in (1).\\
	(4) If $\cP$ has a ladder of at least three steps, then the claim follows similarly as done in (2).	
\end{proof}

\begin{thm}\label{Teorema: P non ha zig-zag se e solo se non ha ladder...}
	Let $\cP$ be a weakly closed path. The following conditions are equivalent:
	\begin{enumerate}
		\item $\cP$ has an $L$-configuration or a ladder of at least three steps or a weak $L$-configuration or a weak ladder;
		\item $\cP$ does not contain zig-zag walks.
	\end{enumerate}
\end{thm}
\begin{proof}
	The sufficient condition follows immediately from Proposition \ref{Proposizione: Se ha una L-conf, weak L-conf ecc, allora P non ha zig-zag}. We prove the necessary one arguing by contradiction. Suppose that $\cP$ has no $L$-configuration, no ladder of at least three steps, no weak $L$-configuration and no weak ladder and we show how it is possible to find a zig-zag walk in $\cP$. We may assume that $v_H$ is respectively the lower right corner of $A_0$ and the upper left corner of $A_1$. Let $\cB_1$ be the maximal horizontal or vertical block of $\cP$ containing $A_1$. We may assume that $\cB_1$ is in horizontal position, because similar arguments hold in the other case. We set $I_1=V(\cB_1)$ and $v_1=v_H$ and $z_1$ as anti-diagonal corners. 
	Let $A_m$ be the cell of $\cP$ such that $[A_1,A_m]=\cB_1$ for some $m\in\{2,\dots,n\}$. For the cell $A_{m+1}$ the following cases are possible:
	\begin{enumerate}
		\item $A_{m+1}$ is at East of $A_m$. It is a contradiction to the maximality of $\cB_1$;
		\item $A_{m+1}$ is at South of $A_m$. Then $\{A_0\}\cup \cB_1\cup \{A_{m+1}\}$ is a weak ladder, so it is a contradiction;
		\item $A_{m+1}$ is at West of $A_m$. Then $A_{m+1}=A_{m-1}$, so it is a contradiction to Definition \ref{Definizione: weakly closed path}.
	\end{enumerate}
	Necessarily $A_{m+1}$ is at North of $A_m$. Now we consider the cell $A_{m+2}$, and we examine its positions with respect to $A_{m+1}$:
	\begin{enumerate}
		\item $A_{m+2}$ is at West of $A_{m+1}$. It is a contradiction to (3) of Definition \ref{Definizione: weakly closed path};
		\item $A_{m+2}$ is at North of $A_{m+1}$. Then $\{A_0\}\cup \cB_1\cup [A_{m+1},A_{m+2}]$ is a weak $L$-configuration if $|\cB_1|=2$ or $\cB_1\cup [A_{m+1},A_{m+2}]$ contains an $L$-configuration if $|\cB_1|>2$, so we have a contradiction in both cases;
		\item $A_{m+2}$ is at South of $A_{m+1}$. Then $A_{m+2}=A_{m}$, so it is a contradiction.
	\end{enumerate}
	Necessarily $A_{m+2}$ is at East of $A_{m+1}$. We observe that the cell $A_{m+3}$ can be at North or at East of $A_{m+2}$. If $A_{m+3}$ is at North of $A_{m+2}$, then by previous arguments $A_{m+4}$ is also at North of $A_{m+3}$, so we denote by $\cB_2$ the maximal vertical block containing $\{A_{m+2},A_{m+3},A_{m+4}\}$ and $V(\cB_1)\cap V(\cB_2)=\{p_1\}$; in such a case we set $I_2=V(\cB_2)$ having $v_2=p_1$ and $z_2$ as diagonal corners. If $A_{m+3}$ is at East of $A_{m+2}$, then $A_{m+4}$ is also at East of $A_{m+3}$, so we denote by $\cB_2$ the maximal horizontal block containing $\{A_{m+2},A_{m+3},A_{m+4}\}$ and $V(\cB_1)\cap V(\cB_2)=\{a_1,b_1\}$, with $a_1<b_1$; in such a case we set $I_2=V(\cB_2\setminus \{A_{m+1}\})$ having $v_2=b_1$ and $z_2$ as diagonal corners. In both cases $|\cB_2|\geq3$. Let $\cB$ be the maximal block containing $A_{n}$ and let $A_p$ be the other extremal cell of $\cB$ for some $p \leq n-1$. If $\cB$ is in vertical position, then $A_{p-1}$ is at East of $A_p$ and $A_{p-2}$ is at North of $A_{p-1}$ by similar arguments. Similarly, if $\cB$ is in horizontal position, then $A_{p-1}$ is at South of $A_p$ and $A_{p-2}$ is at West of $A_{p-1}$. Moreover it is easy to see that $A_{p-2}$ is a cell of a maximal block of $\cP$ of length at least three, denoted by $\cB_f$.\\
	Now, starting from $\cB_2$, we define inductively a sequence of maximal blocks of $\cP$ and, as a consequence, a sequence of inner intervals of $\cP$. 
	Let $\cB_k$ be a maximal block of $\cP$ of length at least three. We may assume that $\cB_k$ is in horizontal position and that there exist $A_{i_k}$ and $A_{i_{k+1}}$ with $i_k<i_{k+1}$ such that $\cB_k=[A_{i_k},A_{i_{k}+1}]$, otherwise we can apply appropriate reflections or rotations. For convenience we set $j=i_{k}+1$. In order to define $\cB_{i+1}$, we distinguish two cases, which depend on the position of $A_{{i_k}-1}$ with respect to $A_{i_k}$. Assume that $A_{i_k-1}$ is at North of $A_{i_k}$ and observe that $A_{i_{k}-2}$ is necessarily at West of $A_{i_{k}-1}$, otherwise $\{A_{i_{k}-2},A_{i_{k}-1}\}\cup \cB_k$ contains an $L$-configuration or Definition \ref{Definizione: weakly closed path} is contradicted. Consider the cell $A_{j+1}$, so $A_{j+1}$ is at North of $A_j$. In fact, if $A_{j+1}$ is at South of $A_j$, then either $\{A_{i_{k}-2},A_{i_{k}-1}\}\cup \cB_k\cup \{A_{j+1},A_{j+2}\}$ is a ladder of three steps or $\cB_k\cup \{A_{j+1},A_{j+2}\}$ contains an $L$-configuration or Definition \ref{Definizione: weakly closed path} is contradicted. By similar arguments we deduce that $A_{j+2}$ is at East of $A_{j+1}$. Now we can define the maximal block $\cB_{k+1}$, depending on the position of $A_{j+3}$.
	\begin{itemize}
	\item If $A_{j+3}$ is at East of $A_{j+2}$, then we denote by $\cB_{k+1}$ the maximal horizontal block of $\cP$ containing $\{A_{j+1},A_{j+2},A_{j+3}\}$. In such a case $|V(\cB_k)\cap V(\cB_{k+1})|=2$ and we set $V(\cB_k)\cap V(\cB_{k+1})=\{a_k,b_k\}$ with $a_k<b_k$. Hence we put $I_{k+1}=V(\cB_{k+1}\setminus \{A_{j+1}\})$ and $v_{k+1}=b_k$, $z_{k+1}$ as diagonal corners (see Figure \ref{Figura:costruzione B_k dimostrazione cond nec e suff no zig-zag 1} (A)).
	\item 	If $A_{j+3}$ is at North of $A_{j+2}$, then $A_{j+4}$ is at North of $A_{j+3}$, otherwise we have a contradiction, since $B_k\cup \{ A_{j+1}, A_{j+2} \} \cup \{ A_{j+3}, A_{j+4} \}$ would be a ladder of three steps. So we denote by $\cB_{k+1}$ the  maximal vertical block of $\cP$ containing $\{A_{j+2},A_{j+3},A_{j+4}\}$. In such a case $|V(\cB_k)\cap V(\cB_{k+1})|=1$ and we set $V(\cB_k)\cap V(\cB_{k+1})=\{p_k\}$. Hence we put $I_{k+1}=V(\cB_{k+1})$ having $v_{k+1}=p_k$, $z_{k+1}$ as diagonal corners (see Figure \ref{Figura:costruzione B_k dimostrazione cond nec e suff no zig-zag 1} (B)).
	\end{itemize}
		\begin{figure}[h!]
	\centering
	\subfloat[]{\includegraphics[scale=0.75]{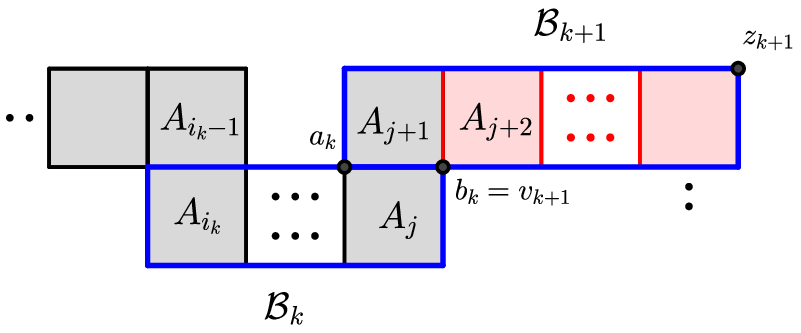}}\qquad
	\subfloat[]{\includegraphics[scale=0.75]{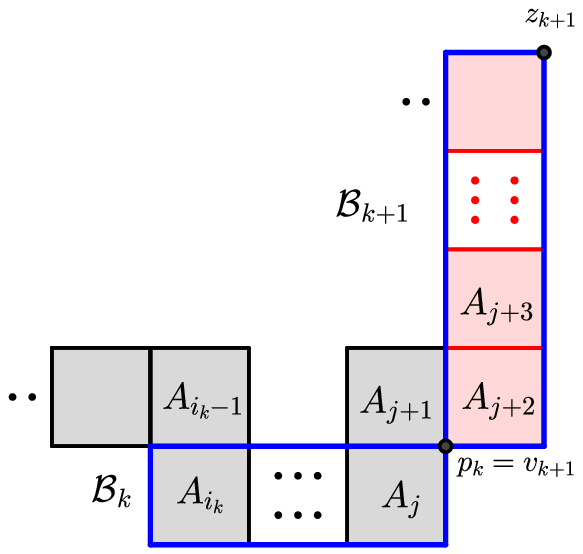}} 
	\caption{}
	\label{Figura:costruzione B_k dimostrazione cond nec e suff no zig-zag 1}
\end{figure} 
	Assume that $A_{{i_k}-1}$ is at South of $A_{i_k}$. By similar arguments, we can define the maximal block $\cB_{k+1}$ and the inner interval $I_{k+1}$, which has $v_{k+1}$ and $z_{k+1}$ as anti-diagonal corners, as done in the previous case (see Figure \ref{Figura:costruzione B_k dimostrazione cond nec e suff no zig-zag 2}).
		\begin{figure}[h!]
		\centering
		\subfloat[]{\includegraphics[scale=0.75]{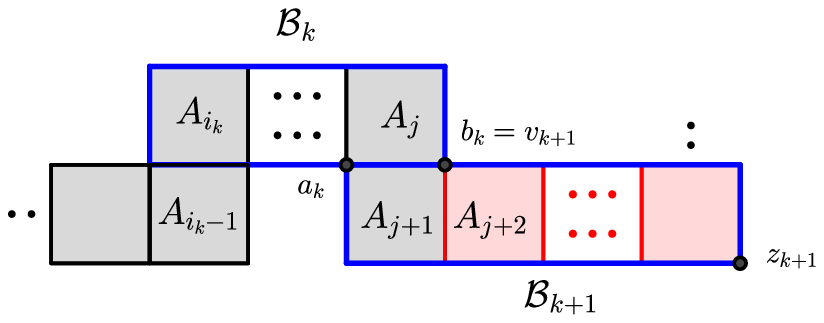}} \qquad
		\subfloat[]{\includegraphics[scale=0.75]{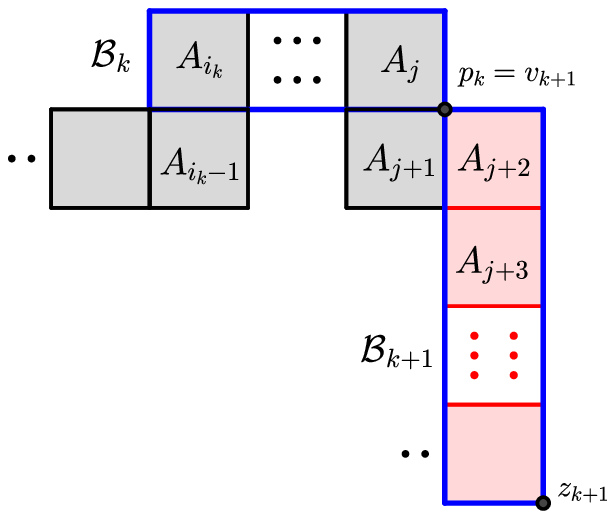}} 
		\caption{}
		\label{Figura:costruzione B_k dimostrazione cond nec e suff no zig-zag 2}
	\end{figure} 
	Observe that there exists a configuration in which $\cB_{k}$ is in horizontal position and $\cB_{k+1}$ is in vertical position, otherwise we have a contradiction to (1) of Definition \ref{Definizione: weakly closed path}. Starting from $k=2$ and using the procedure described before, we define the sequence of maximal block $\cB_2,\cB_3,\dots$ and, since $\cP$ is a weakly closed path, in particular a path from $A_1$ to $A_{p-2}$, then there exists $s\in \N$ such that $\cB_s=\cB_f$. We set $\cB_{s+1}=\cB$ and we observe that the only possible arrangements of the blocks $\cB_s, \cB_{s+1}$ and $\cB_1$ are displayed in Figure \ref{Figura:come si chiudono B_s e B_1 dimostrazione cond nec e suff no zig-zag 1} and \ref{Figura:come si chiudono B_s e B_1 dimostrazione cond nec e suff no zig-zag 2}. In particular, in the configurations of Figure \ref{Figura:come si chiudono B_s e B_1 dimostrazione cond nec e suff no zig-zag 1} we put $I_{s+1}=V(\cB_{s+1})$ having $v_{s+1}=p_s$, $z_{s+1}$ as diagonal corners, and in the configurations of Figure \ref{Figura:come si chiudono B_s e B_1 dimostrazione cond nec e suff no zig-zag 2} we put $I_{s+1}=V(\cB_{s+1}\setminus\{A_p\})$ having $v_{s+1}=b_s$, $z_{s+1}$ as diagonal corners. 
	\begin{figure}[h!]
		\centering
		\subfloat{\includegraphics[scale=0.75]{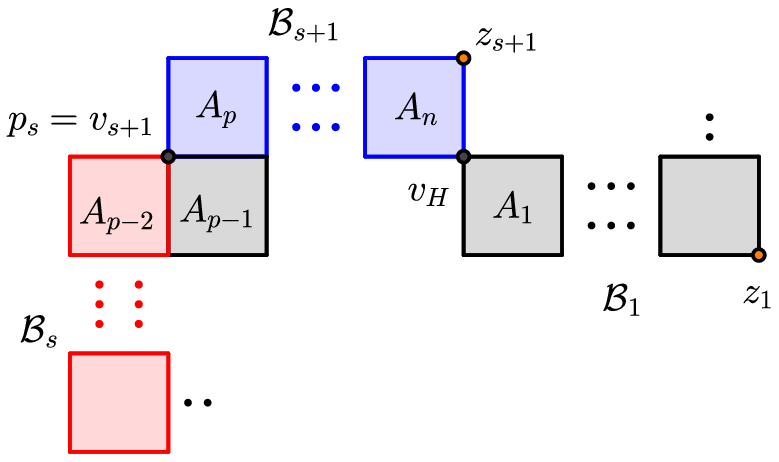}}\qquad
		\subfloat{\includegraphics[scale=0.75]{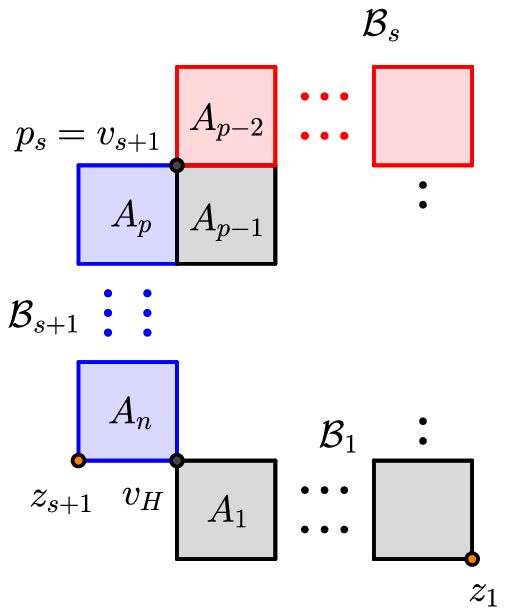}}
		\caption{}
		\label{Figura:come si chiudono B_s e B_1 dimostrazione cond nec e suff no zig-zag 1}
	\end{figure} 
	\begin{figure}[h!]
		\centering\subfloat{\includegraphics[scale=0.75]{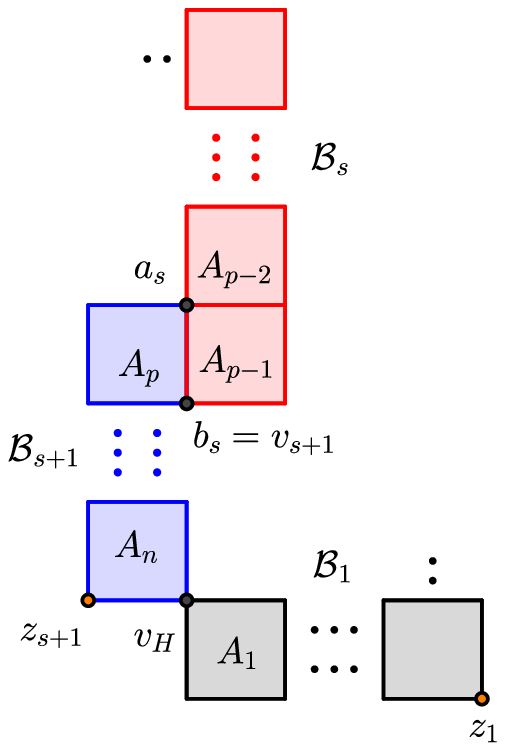}}
		\quad
		\subfloat{\includegraphics[scale=0.75]{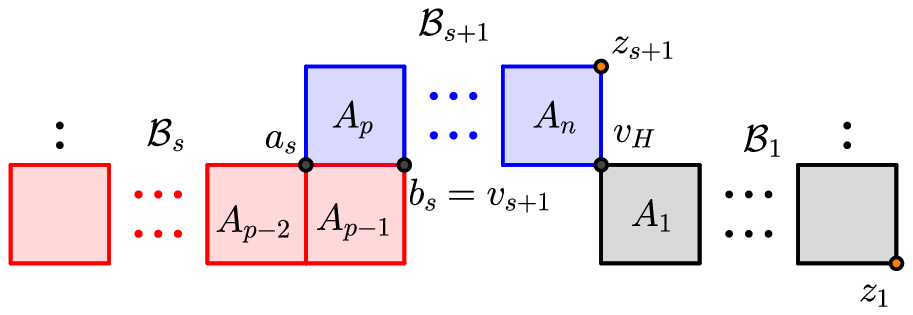}}
		\caption{}
		\label{Figura:come si chiudono B_s e B_1 dimostrazione cond nec e suff no zig-zag 2}
	\end{figure} 

\noindent Hence there exists a sequence of maximal blocks $\cB_1,\dots,\cB_s,\cB_{s+1}$ of $\cP$ with $V(\cB_1)\cap V(\cB_{s+1})=\{v_H\}$ and a sequence $I_1,I_2,\dots,I_s,I_{s+1}$ of inner intervals of $\cP$ with $I_k\subseteq V(\cB_k)$ for all $k=1,\dots,s+1$, having the properties described before. We prove that $\cW:I_1,\dots,I_s, I_{s+1}$ is a zig-zag walk of $\cP$. 
	\begin{enumerate}
		\item It is clear by the previous construction that $I_{s+1}\cap I_1=\{v_H\}$ and $I_k\cap I_{k+1}=\{v_{k+1}\}$ for all $k=1,\dots,s$.
		\item Let $k\in\{1,\dots,s\}$. Firstly suppose that $k\in\{2,3,\dots,s-1\}$. Consider the blocks $\cB_{k-1}$, $\cB_{k}$ and $\cB_{k+1}$. We may assume that $\cB_{k-1}$ is in horizontal position and that there exist $A_{i_k}$ and $A_{i_{k+1}}$ with $i_k<i_{k+1}$ such that $\cB_{k-1}=[A_{i_k},A_{i_{k+1}}]$ and $A_{{i_k}-1}$ is at North of $A_{i_k}$, otherwise we can do opportune reflections or rotations. Assume that $\cB_{k}$ is in horizontal position. By the construction of $\cB_k$ and $\cB_{k+1}$, we have the situation described in Figure \ref{Figura:dimostrazione cond nec e suff no zig-zag v_k stanno sullo stesso edge interval} (A), where the dashed lines indicate the block $\cB_{k+1}$ depending on its position. Therefore it follows that $v_k$ and $v_{k+1}$ are on the same edge interval of $\cP$. The same holds if $\cB_k$ is in vertical position, in particular see Figure \ref{Figura:dimostrazione cond nec e suff no zig-zag v_k stanno sullo stesso edge interval} (B).
		 \begin{figure}[h!]
			\centering\subfloat[]{\includegraphics[scale=0.6]{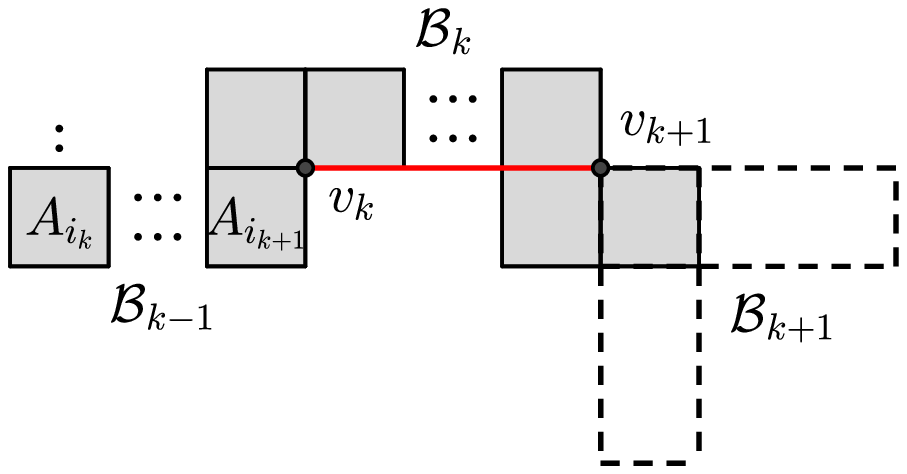}}
			\qquad\qquad
			\subfloat[]{\includegraphics[scale=0.6]{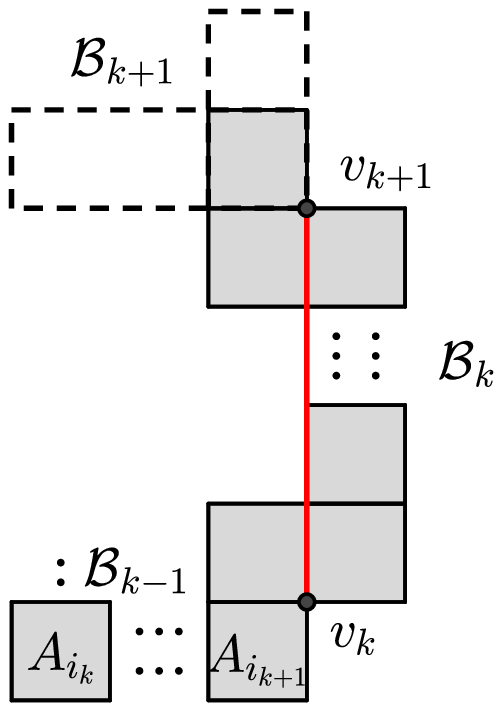}}
			\caption{}
			\label{Figura:dimostrazione cond nec e suff no zig-zag v_k stanno sullo stesso edge interval}
		\end{figure}
	If $k=1$ or $k=s$, then we can consider the blocks $\cB_{s}$, $\cB_{s+1}$ and $\cB_{1}$, so with reference to the Figures \ref{Figura:come si chiudono B_s e B_1 dimostrazione cond nec e suff no zig-zag 1} and \ref{Figura:come si chiudono B_s e B_1 dimostrazione cond nec e suff no zig-zag 2} the desired claim follows. 		
	\item Let $k,j\in\{1,\dots,s+1\}$ with $k< j$. If $k=1$ and $j=s+1$ we have the situation in Figure \ref{Figura:come si chiudono B_s e B_1 dimostrazione cond nec e suff no zig-zag 1} or \ref{Figura:come si chiudono B_s e B_1 dimostrazione cond nec e suff no zig-zag 2}, so the interval having $z_1$ and $z_{s+1}$ as corners is not an inner interval of $\cP$. If $j=k+1$, then we consider the blocks $\cB_k$ and $\cB_{k+1}$. We may assume that $\cB_{k}$ is in horizontal position and that there exist $A_{i_k}$ and $A_{i_{k+1}}$ with $i_k<i_{k+1}$ such that $\cB_k=[A_{i_k},A_{i_{k}+1}]$ and $A_{{i_k}-1}$ is at South of $A_{i_k}$, otherwise we can do appropriate reflections or rotations. $\cB_{k+1}$ is either in horizontal or vertical position. With reference to Figure \ref{Figura: dimostrazione cond nec e suff no zig-zag z_k non formano inner interval}, in both cases the interval having $z_k$ and $z_{k+1}$ as corners is not an inner interval of $\cP$. If $j\neq k+1,k-1$, then the desired conclusion follows.
	\begin{figure}[h]
		\centering
		\includegraphics[scale=0.6]{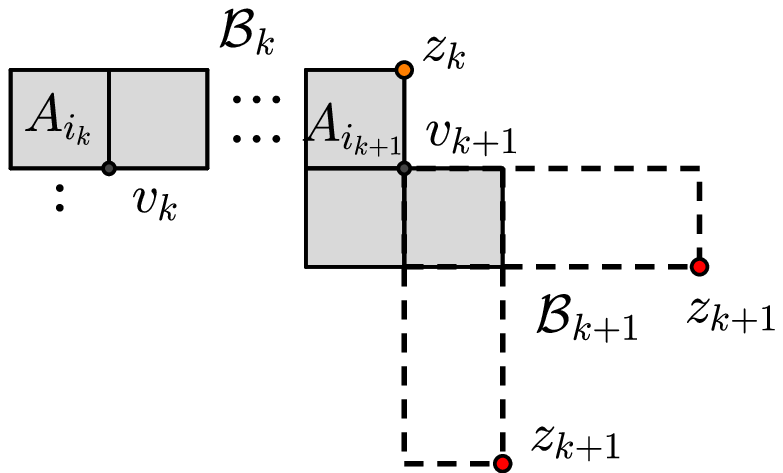}
		\caption{}
		\label{Figura: dimostrazione cond nec e suff no zig-zag z_k non formano inner interval}
	\end{figure}
	\end{enumerate}
	Therefore $\cW:I_1,\dots, I_s$ is a zig-zag walk in $\cP$.
 \end{proof}

\begin{defn}\rm\label{Definizione: Toric ideal}
	 Let $\cP$ be a non-simple polyomino with a unique hole $\cH$. Let $\{V_i\}_{i\in I}$ be the set of the maximal edge intervals of $\cP$ and $\{H_j\}_{j\in J}$ be the set of the maximal horizontal edge intervals of $\cP$. Let $\{v_i\}_{i\in I}$ and $\{h_j\}_{j\in J}$ be the set of the variables associated respectively to $\{V_i\}_{i\in I}$ and $\{H_j\}_{j\in J}$. Let $\cH$ be the hole of $\cP$ and $w$ be another variable. Let $\cI$ be a subset of $V(\cP)$ and we define the following map:
	\begin{align*}
	\alpha: V(\cP)&\longrightarrow K[\{v_i,h_j,w\}:i\in I,j\in J]\\
	a&\longmapsto v_ih_jw^k
	\end{align*}
	where $a\in V_i\cap H_j$, $k=0$ if $a\notin \cI$, and $k=1$ if $a\in \cI$. The toric ring, denoted by $T_{\cP}$, is $K[\alpha(v):v\in V(\cP)]$. We consider the following surjective ring homomorphism $\phi: S \longrightarrow T_{\cP}$ defined by $\phi(x_a)=\alpha(a)$ and the kernel of $\phi$ is the toric ideal denoted by $J_{\cP}$.
\end{defn}

\begin{prop}\label{Proposizione: Se ha ecc... allra I=J}
	Let $\cP$ be a weakly closed path. If one of the following conditions holds:
	\begin{enumerate}
	    \item $\cP$ has an $L$-configuration,
		\item $\cP$ has a weak $L$-configuration,
		\item $\cP$ has a ladder of at least three steps,
		\item $\cP$ has a weak ladder,
	\end{enumerate}
	then $I_{\cP}$ is prime.
\end{prop}
\begin{proof}
(1) Suppose that $\cP$ has an $L$-configuration $\{C_1,C_2,C_3,C_4,C_5\}$ and we can consider suitable reflections or rotations of $\cP$ in order to have the $L$-configuration in the position of Figure \ref{Figura:L-conf in toric inizio sezione}. For convenience set $C_3=A$.
\begin{figure}[h!]
	\centering
	\includegraphics[scale=0.9]{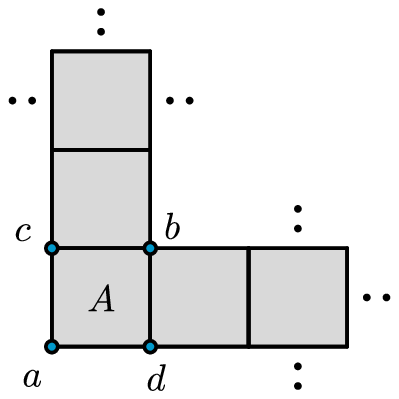} 
	\caption{}
	\label{Figura:L-conf in toric inizio sezione}
\end{figure}
In order to have the desired claim, it is sufficient to prove that $I_{\cP}=J_{\cP}$, where $J_{\cP}$ is the toric ideal defined in Definition \ref{Definizione: Toric ideal} with $\cI=V(A)$. Observe that it follows easily by arguing as done in [\cite{Cisto_Navarra}, Theorem 4.2] and by applying Theorem \ref{Teorema: P collezione semplice dedebolmente connessa allora I=J}.\\
(2) Let $\cL=\{C\}\cup [A,B]\cup [D,F]$ be a weak $L$-configuration of $\cP$. We consider opportune reflections or rotations of $\cP$ in order to have $\cL$ as in the picture on the left in Figure \ref{Figura:esempio weak L configuration}. By similar arguments as in case (1), we conclude that $I_{\cP}=J_{\cP}$, where $J_{\cP}$ is the toric ideal defined in Definition \ref{Definizione: Toric ideal} with $\cI=V(B)$.\\
(3) Suppose that $\cP$ has a ladder of at least three steps and let $\cB=\{\cB_i\}_{i=1,\dots,m}$ be a maximal ladder of $m$ steps, with $m>2$. We can consider suitable reflections or rotations of $\cP$ in order to have the ladder as in Figure \ref{Figura:ladder in toric inizio sezione}. We may assume that the block $\cB_{m-1}$ consists of $n$ cells which we denote by $B_1,\dots, B_n$ from left to right. Let $b_i$ be the lower left corner of $B_i$ for all $i=1,\dots,n$ and let $B$ be the cell of $\cB_m$, having an edge in common with $B_n$. We denote by $a,b$ the diagonal corners of $B$ and by $d$ the other anti-diagonal corner.
\begin{figure}[h!]
	\centering
	\includegraphics[scale=0.8]{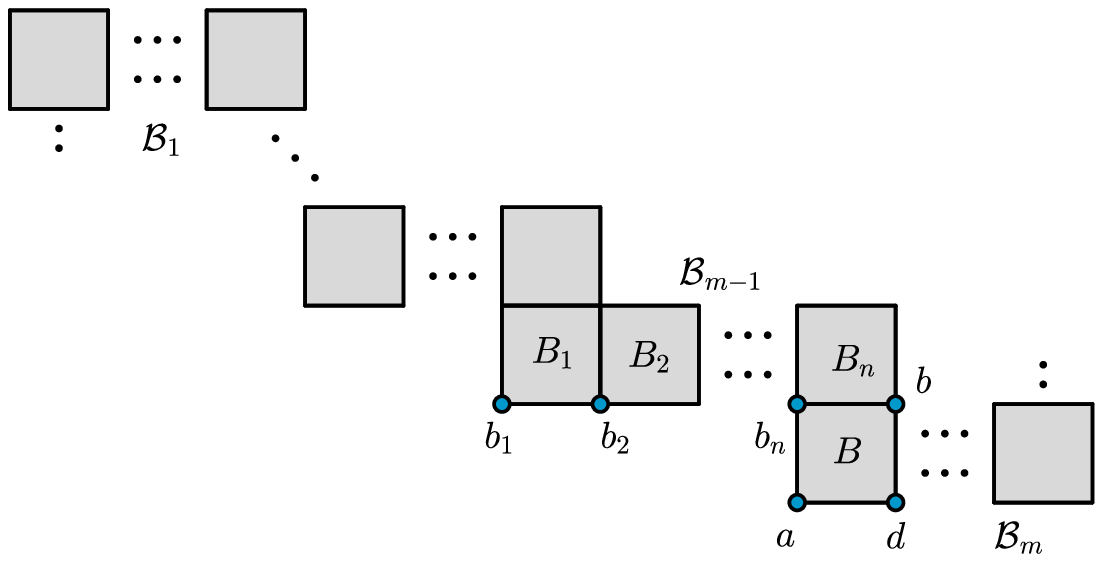} 
	\caption{}
	\label{Figura:ladder in toric inizio sezione}
\end{figure}
	We want to show that $I_{\cP}=J_{\cP}$, where $J_{\cP}$ is the toric ideal defined in \ref{Definizione: Toric ideal} with $\cI=\{b_1,\dots,b_n,b,a,d\}$. We observe that it follows by arguing as done in [\cite{Cisto_Navarra}. Theorem 5.2] and by using Theorem \ref{Teorema: P collezione semplice dedebolmente connessa allora I=J}.\\
	(4) Suppose that $\cP$ has a weak ladder. Let $\cL=\{C,D\}\cup [A,B]$ be the weak ladder of $\cP$, that we can assume being as in Figure \ref{Figura:esempio weak ladder}. We may assume that the block $[A,B]$ is made up of $n$ cells, with $n\geq 2$, which we denote by $C_1,\dots, C_n$ from left to right. Firstly, we assume that the block containing $C$ is in vertical position. We denote by $l_C$ and $l_{C_1}$ respectively the lower left corner of $C$ and $C_1$ (see Figure \ref{Figura:dimostrazione toric weak ladder} (A)).
\begin{figure}[h!]
			\centering
			\subfloat[]{\includegraphics[scale=0.8]{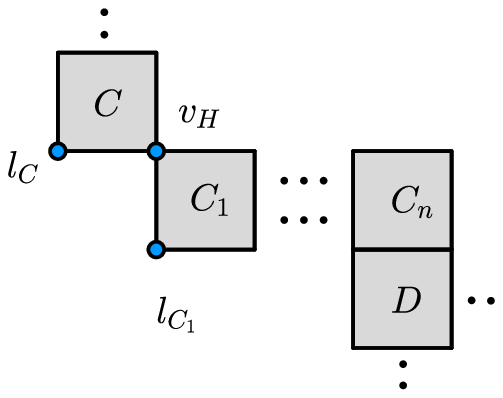}}\qquad \qquad
			\subfloat[]{\includegraphics[scale=0.8]{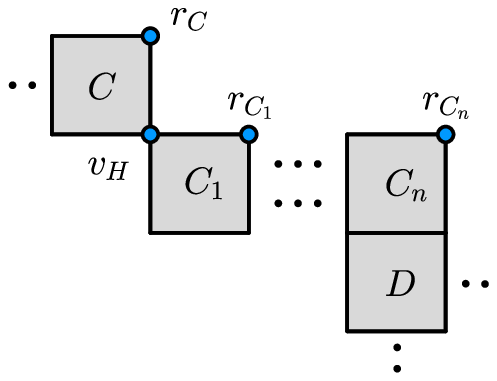}}
			\caption{}
			\label{Figura:dimostrazione toric weak ladder}
		\end{figure}	
By similar arguments as in case (1), we conclude that $I_{\cP}=J_{\cP}$, where $J_{\cP}$ is the toric ideal defined at the beginning of this section with $\cI=\{v_H,l_C,l_{C_1}\}$. 
	Now, we assume that the block containing $C$ is in horizontal position. We denote by $r_C$ the upper right corner of $C$ and by $r_{C_i}$ the upper right corner of $C_i$ for all $i=1,\dots,n$. The conclusion $I_{\cP}=J_{\cP}$, where $J_{\cP}$ is the toric ideal defined in Definition \ref{Definizione: Toric ideal} with $\cI=\{v_H,r_C,r_{C_1},\dots,r_{C_{n}}\}$, follows as in the proof of case (3) (see Figure \ref{Figura:dimostrazione toric weak ladder} (B)).
\end{proof}

\begin{rmk}\rm\label{Remark: dopo la dim coll non sempl e deb connessa primo}
	We observe that the weakly connected and non-simple collection of cells $\cP$ in Figure \ref{Figura:collezione di celle non semplici e non prima} (B) is prime in fact by similar arguments as in case (1) of Proposition \ref{Proposizione: Se ha ecc... allra I=J} it follows that $I_{\cP}=J_{\cP}$ where $J_{\cP}$ is the toric ideal defined in Definition \ref{Definizione: Toric ideal} with $\cI=V(A)$.
\end{rmk}

\begin{thm}
	Let $\cP$ be a weakly closed path. $I_{\cP}$ is prime if and only if $\cP$ does not contain zig-zag walks.
\end{thm}
\begin{proof}
	The necessary condition is proved in \cite[Corollary 3.6]{Trento}. The sufficient one follows from Theorem \ref{Teorema: P non ha zig-zag se e solo se non ha ladder...} and from Proposition \ref{Proposizione: Se ha ecc... allra I=J}.
\end{proof}

\noindent\textbf{Acknowledgements.} The second author dedicates this paper to the memory of his father Sebastiano Navarra (13/11/1950 - 03/04/2021).

	\end{document}